\documentclass[a4paper, reqno, 12pt]{amsart}

\usepackage[usenames,dvipsnames]{color}
\usepackage{amsthm,amsfonts,amssymb,amsmath,amsxtra}
\usepackage[all]{xy}
\SelectTips{cm}{}
\usepackage{xr-hyper}
\usepackage[colorlinks=
   citecolor=Black,
   linkcolor=Red,
   urlcolor=Blue]{hyperref}
\usepackage{verbatim}

\usepackage{mathrsfs}

\usepackage{mathtools}

\RequirePackage{xspace}
\RequirePackage{etoolbox}
\RequirePackage{varwidth}
\RequirePackage{enumitem}
\RequirePackage{tensor}
\RequirePackage{mathtools}
\RequirePackage{longtable}
\RequirePackage{multirow}

\def\ge{\geqslant}
\def\le{\leqslant}

\def\d{\delta}
\def\D{\Delta}

\def\e{\epsilon}

\def\t{\tau}

\def\i{^{-1}}

\def\<{\langle}
\def\>{\rangle}

\newcommand{\bG}{\mathbf G}

\newcommand{{\BG}}{\ensuremath{\mathbb {G}}\xspace}

\newcommand{{\BK}}{\ensuremath{\mathbb {K}}\xspace}

\newcommand{\BZ}{\ensuremath{\mathbb {Z}}\xspace}

\newcommand{\CL}{\ensuremath{\mathcal {L}}\xspace}

\newcommand{\CP}{\ensuremath{\mathcal {P}}\xspace}

\newcommand{\CR}{\ensuremath{\mathcal {R}}\xspace}

\newcommand{\CT}{\ensuremath{\mathcal {T}}\xspace}

\newcommand{\CY}{\ensuremath{\mathcal {Y}}\xspace}

\newcommand{\Ad}{{\mathrm{Ad}}}

\newcommand{\Sh}{\mathrm{Sh}}

\newcommand{\id}{\ensuremath{\mathrm{id}}\xspace}

\newcommand{\Ind}{{\mathrm{Ind}}}

\newcommand{\ds}{\displaystyle}

\def\tW{\tilde W}

\def\kk{\mathbf k}

\newtheorem{theorem}{Theorem}
\newtheorem{proposition}[theorem]{Proposition}
\newtheorem{lemma}[theorem]{Lemma}

\newtheorem{corollary}[theorem]{Corollary}

\theoremstyle{definition}

\newtheorem{example}[theorem]{Example}

\newtheorem{remark}[theorem]{Remark}

\numberwithin{equation}{section}
\numberwithin{theorem}{section}

\setitemize[0]{leftmargin=*,itemsep=\the\smallskipamount}
\setenumerate[0]{leftmargin=*,itemsep=\the\smallskipamount}

\renewcommand{\to}{%
   \ifbool{@display}{\longrightarrow}{\rightarrow}%
   }
\let\shortmapsto\mapsto
\renewcommand{\mapsto}{%
   \ifbool{@display}{\longmapsto}{\shortmapsto}%
   }
\newlength{\olen}
\newlength{\ulen}
\newlength{\xlen}
\newcommand{\xra}[2][]{%
   \ifbool{@display}%
      {\settowidth{\olen}{$\overset{#2}{\longrightarrow}$}%
       \settowidth{\ulen}{$\underset{#1}{\longrightarrow}$}%
       \settowidth{\xlen}{$\xrightarrow[#1]{#2}$}%
       \ifdimgreater{\olen}{\xlen}%
          {\underset{#1}{\overset{#2}{\longrightarrow}}}%
          {\ifdimgreater{\ulen}{\xlen}%
             {\underset{#1}{\overset{#2}{\longrightarrow}}}
             {\xrightarrow[#1]{#2}}}}%
      {\xrightarrow[#1]{#2}}
   }
\makeatother
\newcommand{\xyra}[2][]{%
   \settowidth{\xlen}{$\xrightarrow[#1]{#2}$}%
   \ifbool{@display}%
      {\settowidth{\olen}{$\overset{#2}{\longrightarrow}$}%
       \settowidth{\ulen}{$\underset{#1}{\longrightarrow}$}%
       \ifdimgreater{\olen}{\xlen}%
          {\mathrel{\xymatrix@M=.12ex@C=3.2ex{\ar[r]^-{#2}_-{#1} &}}}%
          {\ifdimgreater{\ulen}{\xlen}%
             {\mathrel{\xymatrix@M=.12ex@C=3.2ex{\ar[r]^-{#2}_-{#1} &}}}
             {\mathrel{\xymatrix@M=.12ex@C=\the\xlen{\ar[r]^-{#2}_-{#1} &}}}}}%
      {\mathrel{\xymatrix@M=.12ex@C=\the\xlen{\ar[r]^-{#2}_-{#1} &}}}%
   }
\makeatletter
\newcommand{\xla}[2][]{%
   \ifbool{@display}%
      {\settowidth{\olen}{$\overset{#2}{\longleftarrow}$}%
       \settowidth{\ulen}{$\underset{#1}{\longleftarrow}$}%
       \settowidth{\xlen}{$\xleftarrow[#1]{#2}$}%
       \ifdimgreater{\olen}{\xlen}%
          {\underset{#1}{\overset{#2}{\longleftarrow}}}%
          {\ifdimgreater{\ulen}{\xlen}%
             {\underset{#1}{\overset{#2}{\longleftarrow}}}
             {\xleftarrow[#1]{#2}}}}%
      {\xleftarrow[#1]{#2}}
   }
\newcommand{\isoarrow}{%
   \ifbool{@display}{\overset{\sim}{\longrightarrow}}{\xrightarrow\sim}%
   }
   
\begin{document}

\title[]{A generalization of cyclic shift classes}

\author[Xuhua He]{Xuhua He}
\address{The Institute of Mathematical Sciences and Department of Mathematics, The Chinese University of Hong Kong, Shatin, N.T., Hong Kong SAR, China.}
\email{xuhuahe@math.cuhk.edu.hk}


%\date{\today}

\begin{abstract}
Motivated by Lusztig's $G$-stable pieces, we consider the combinatorial pieces: the pairs $(w, K)$ for elements $w$ in the Weyl group and subsets $K$ of simple reflections that are normalized by $w$. We generalize the notion of cyclic shift classes on the Weyl groups to the set of combinatorial pieces. We show that the partial cyclic shift classes of combinatorial pieces associated with minimal-length elements have nice representatives. As applications, we prove the left-right symmetry and the compatibility of the induction functors of the parabolic character sheaves. 
\end{abstract}

\maketitle

\section*{Introduction}

\subsection{Cyclic shifts of Weyl group elements} Let $G$ be a connected reductive group over an algebraically closed field. Let $B$ be a Borel subgroup of $G$ and $W$ be the Weyl group of $G$. Then we have the Bruhat decomposition $G=\sqcup_{w \in W} B \dot w B$. A central theme in geometric representation theory is understanding the (many) mysterious relationships between the reductive group $G$ and its Weyl group $W$ (see, e.g., the Kazhdan-Lusztig theory and the Springer theory). 

We are particularly interested in the conjugation action on $W$. Note that $W$ is a Coxeter group and is generated by simple reflections. In particular, any two elements in the same conjugacy class of $W$ are conjugated by a sequence of simple reflections. We say that two elements are in the same cyclic shift class if they are conjugated by a sequence of simple reflection, and each conjugation preserves the length. This length-preserving condition is crucial in the connection to the conjugation problems on $G$. If $w$ and $w'$ are in the same cyclic shift class, then we have an isomorphism between the quotient stacks \[\tag{a} \frac{B \dot w B}{\D(B)} \cong \frac{B \dot w' B}{\D(B)}.\] Here $\D(B)$ denotes the conjugation action. Such an isomorphism plays an important role in the interaction between the conjugacy classes of $W$ and unipotent conjugacy classes of $G$ (see \cite{Lu}).

\subsection{Cyclic shifts of combinatorial pieces} In this paper, we consider an analogy of (a) for parabolic subgroups. Such an isomorphism will play a role in the study of Lusztig's theory of parabolic character sheaves, which we will discuss in \S\ref{sec:0.4}. 

The object we consider is the quotient stack $\frac{P \dot w P}{\D(P)}$, where $P=L U_P$ is a standard parabolic subgroup of $G$, and $w$ is an element in the Weyl group which normalizes the standard Levi subgroup $L$. In fact, for the applications, the object we consider will be more technical than $\frac{P \dot w P}{\D(P)}$, but we ignore such complexity in the introduction. 

Our first goal is to establish isomorphisms between $\frac{P \dot w P}{\D(P)}$ for various pairs $(w, P)$. Note that the conjugation by simple reflections does not send standard Levi subgroups to Levi subgroups. To overcome the difficulty, we introduce combinatorial pieces. A combinatorial piece is a pair $(w, K)$, where $w$ is an element in the Weyl group, and $K$ is a subset of simple reflections (which represents a standard parabolic subgroup $P_K$ of $G$) such that $w$ normalizes $K$. We then introduce the cyclic shift classes of the combinatorial pieces in \S\ref{sec:def} and prove in Proposition \ref{prop:conj} that if $(w, K)$ and $(w', K')$ are in the same cyclic shift class, then we have an isomorphism between the quotient stacks \[\tag{b} \frac{P_K \dot w P_K}{\D(P_K)} \cong \frac{P_{K'} \dot w' P_{K'}}{\D(P_{K'})}.\]

Note that we may regard the Weyl group $W$ as a subset of the combinatorial pieces via the map $w \mapsto (w, \emptyset)$. The combinatorial pieces $(w, \emptyset)$ and $(w', \emptyset)$ are cyclic shift if and only if $w$ and $w'$ are cyclic shift in the usual sense. In this case, $P_{\emptyset}=B$. Thus (b) can be regarded as a generalization of (a). 

\subsection{Representatives} Our next goal is to obtain nice representatives for certain cyclic shift classes of combinatorial pieces. 

Let $J$ be a subset of the simple reflections and $W_J$ be the standard parabolic subgroup of $W$ generated by $J$. We consider the partial conjugacy classes on $W$, i.e., the orbits of the conjugation by $W_J$ on $W$. By \cite{He07}, we have the decomposition \[\tag{c} \frac{W}{\D(W_J)}=\bigsqcup_{(w, K)} \frac{W_K w}{\D(W_K)},\] where $\frac{W}{\D(W_J)}$ denotes the quotient stack of $W$ by the conjugation action of $W_J$, and the right hand side runs over the pairs $(w, K)$, where $w$ are minimal length representatives in the right cosets $W_J \backslash W$ and $K$ is the maximal subset of $J$ such that $(w, K)$ is a combinatorial piece. Moreover, if $w$ is a minimal length element in its partial conjugacy class, then $w$ is cyclic shift (via partial conjugation, i.e., a sequence of conjugations in $W_J$) to an element in $W_K w$ for some $(w, K)$ occurring in the right side of (c). 

We prove an ``upgraded version'' of the above statement in Theorem \ref{thm:cyc}. Namely, if $w_1$ is a minimal length element in its partial conjugacy class and $(w_1, K_1)$ is a combinatorial piece with $K_1 \subset J$, then \[\tag{d} \text{$(w_1, K_1)$ is cyclic shift (via partial conjugation) to $(w_2, K_2)$},\] where $w_2$ is an element in $W_K w$ for some $(w, K)$ occurring in the right side of (c). 

\subsection{Applications to parabolic character sheaves}\label{sec:0.4}

Now we discuss some applications. Lusztig introduced the parabolic character sheaves in \cite{Lu1}. Roughly speaking, the parabolic character sheaves are certain simple perverse sheaves on the quotient stack $\CY_J=\frac{U_{P_J} \backslash G/U_{P_J}}{\D(L_J)}$. The parabolic character sheaves rely on the following decomposition of $\CY_J$ into $G$-stable pieces: 
\[
\tag{e} \CY_J=\bigsqcup_{w} \CY_{J, w}, \text{ where } \CY_{J, w}``=" \frac{P_K \dot w P_K}{\D(P_K)},
\]
Here $w$ runs over the minimal length representatives in the right coset $W_J \backslash W$. And the quotation mark means we ignore the gerbs for unipotent groups. We refer to \S\ref{sec:zjw} for the precise statement. 

It is worth pointing out that if $(w, K)$ is conjugate to $(w', K')$, then it is easy to construct an isomorphism $\frac{L_K \dot w}{\D(L_K)} \cong \frac{L_{K'} \dot w'}{\D(L_{K'})}$. However, it is more involved to relate $\frac{P_K \dot w P_K}{\D(P_K)}$ for different combinatorial pieces $(w, K)$. The results (b) and (d) on the cyclic shifts of combinatorial pieces allow us to understand certain behavior of the sheaves associated with an arbitrary combinatorial piece with respect to the decomposition (e). 

We give two applications in \S\ref{sec:app}.

(1) The left-right symmetry. 

Note that in the decompositions (c) and (d), we use the right cosets $W_J \backslash W$. One may define similar decompositions using the left cosets $W/W_J$. A priori, the decompositions obtained using the left and right cosets might differ. However, using the cyclic shifts of combinatorial pieces, we show that there exists a symmetry between the left cosets and the right cosets such that the corresponding substacks in $\frac{W}{\D(W_J)}$ and $\CY_J$ coincide. 

(2) Induction functors. 

Let $J \subset J'$. Lusztig defined the functor $f: \Sh(\CY_J) \to \Sh(\CY_{J'})$. Suppose that $w$ (resp. $w'$) is the minimal length representative in its right coset $W_J \backslash W$ (resp. $W_{J'} \backslash W$), $w$ and $w'$ are in the same $W_{J'}$-conjugacy class and $\ell(w)=\ell(w')$. Then we have $$\CY_{J, w}``=" \frac{P_K \dot w P_K}{\D(P_K)} \text{ and } \CY_{J', w'}``=" \frac{P_{K'} \dot w' P_{K'}}{\D(P_{K'})}.$$ In Theorem \ref{thm:main}, we show that the restriction of the functor $f$ to $\CY_{J, w}$ and $\CY_{J', w'}$ is essentially the Harish-Chandra induction functor associated to the reductive group $L_{K'}$ and its Levi subgroups. This result is used in the work of Li, Nadler, and Yun \cite{LNY}. 

\smallskip

{\bf Acknowledgement: } XH is partially supported by funds connected with Choh-Ming Chair at CUHK, by Hong Kong RGC grant 14300221, and by the Xplorer prize. This paper is motivated by a question of Zhiwei Yun on the parabolic character sheaves, which is used in \cite{LNY}. We thank him for explaining to me the question and related materials in \cite{LNY}. We also thank George Lusztig for the helpful discussions.

%\tableofcontents
\section{Cyclic shift classes}

\subsection{Conjugacy graph}\label{sec:cyc} Let $W$ be a Weyl group and $S$ be the set of simple reflections. Let $\ell$ be the length function on $W$ and $\le$ be the Bruhat order on $W$. Let $\d$ be a group automorphism on $W$, which preserves $S$. We consider the $\d$-conjugation action on $W$ defined by $w \cdot_\d w'=w w' \d(w) \i$. 

Following \cite[\S 3.2]{GP00}, we define the {\it $\d$-conjugacy graph} as follows. It is the direct graph with vertices in $W$ and the edges are of the form $w \xrightarrow{s}_\d w'$, where $w, w' \in W$, $s \in S$ such that $w'=s w \d(s)$ and $\ell(w') \le \ell(w)$. 

We denote by $\to_\d$ the pre-order relation induced on $W$ by
the $\d$-conjugacy graph. In other words, for $w, w' \in W$, $w \to_\d w'$ if there exists a sequence of elements $w=w_1, w_2, \ldots, w_n=w$ and $s_1, \ldots, s_{n-1} \in S$ such that $w_i \xrightarrow{s_i}_\d w_{i+1}$ for $1 \le i \le n-1$. 

We write $w \approx_\d w'$ if $w \to_\d w'$ and $w' \to_\d w$. It is easy to see that $\approx_\d$ gives an equivalence relation on $W$. For $w \in W$, we write $\text{Cyc}_\d(w)=\{w' \in W; w \approx_\d w'\}$ and call it the {\it $\d$-cyclic shift class } of $w$. 

If $\d$ is the identity map, we may ignore the subscript $\d$. 

Now we give an example of the conjugacy graph. 

\begin{example}\label{eg:1}
Let $W=S_4$ be the finite Weyl group of type $A_3$ with the set of simple reflections $S=\{s_1, s_2, s_3\}$. We simply write $s_{a b c \cdots}$ instead of $s_a s_b s_c \cdots$. It is easy to see that the conjugacy class of $s_1 s_2 s_3$ forms a connected component of the conjugacy graph. This connected component is listed below. 
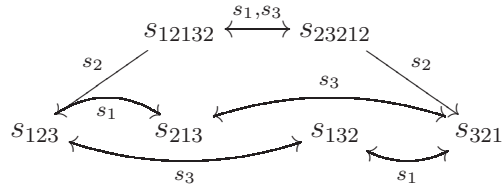
\[
\xymatrix{
& s_{12132} \ar[ld]_{s_2} \ar[r]^{s_1, s_3} & s_{23212} \ar[rd]^{s_2} \ar[l] & \\
s_{123} \ar@/^1pc/[r]_{s_1} \ar@/_1pc/[rr]_{s_3} & s_{213} \ar@/^1pc/[rr]^{s_3} & s_{132} \ar@/^1pc/[ll] \ar@/_1pc/[r]_{s_1} & s_{321} \ar@/^1pc/[l] \ar@/_1pc/[ll].
}
\]
\end{example}

\subsection{Geometric cyclic shifts} Let $\bG$ be a connected reductive group over an algebraically closed field $\kk$. Let $G=\bG(\kk)$ and $B$ be a Borel subgroup of $G$ with maximal torus $T$ and unipotent radical $U$. Let $\CR$ be the root datum of $G$ and $\d$ be an automorphism of $\CR$. In particular, $\d(B)=B$. We still denote by $\d$ for the corresponding automorphism on $G$ and on its Weyl group $W=N_G(T)/T$. For any $w \in W$, we choose a representative $\dot w$ in $N_G(T)$. 

Consider the $\d$-conjugation action of $G$ defined by $g \cdot_\d g'=g g' \d(g) \i$. Let $\frac{G}{\Ad_\d(G)}$ be the quotient stack. We have the Bruhat decomposition $G=\sqcup_{w \in W} B \dot w B$. Each Bruhat cell $B \dot w B$ is stable under the $\d$-conjugation action by $B$. Let $\frac{B \dot w B}{\Ad_\d(B)}$ be the corresponding quotient stack and $$\pi_w: \frac{B \dot w B}{\Ad_\d(B)} \to \frac{G}{\Ad_\d(G)}$$ be the natural morphism. 

On the other hand, we have the automorphism on $T$ defined by $t \mapsto \dot w \d(t) \dot w \i$. Let $\frac{T}{\Ad_{\dot w \circ \d}(T)}$ be the quotient stack and $$p_w: \frac{B \dot w B}{\Ad_\d(B)} \to \frac{T}{\Ad_{\dot w \circ \d}(T)}$$ be the natural morphism. 

We have the following (upgraded) geometric version of cyclic shifts. 

\begin{proposition}\label{prop:cyc}
Let $w, w' \in W$. Suppose that $w \approx_\d w'$. Then there exist isomorphisms $\frac{B \dot w B}{\Ad_\d(B)} \cong \frac{B \dot w' B}{\Ad_\d(B)}$ and $\frac{T}{\Ad_{\dot w \circ \d}(T)} \cong  \frac{T}{\Ad_{\dot w' \circ \d}(T)}$ such that the following diagram commutes
\[
\xymatrix{
\frac{T}{\Ad_{\dot w \circ \d}(T)} \ar@{-->}[d]^\cong & \frac{B \dot w B}{\Ad_\d(B)} \ar[l]_-{p_w} \ar[r]^-{\pi_{w}} \ar@{-->}[d]^\cong & \frac{G}{\Ad_\d(G)} \ar@{=}[d] \\
\frac{T}{\Ad_{\dot w' \circ \d}(T)} & \frac{B \dot w' B}{\Ad_\d(B)} \ar[l]_-{p_{w'}} \ar[r]^-{\pi_{w'}} & \frac{G}{\Ad_\d(G)}.
}
\]
\end{proposition}

\begin{remark}
(1) An isomorphism $\frac{B \dot w B}{\Ad_\d(B)} \to \frac{B \dot w' B}{\Ad_\d(B)}$ was constructed in \cite[\S 3.4]{HL}. We follow the proof in loc.cit. 

(2) The isomorphisms $\frac{B \dot w B}{\Ad_\d(B)} \cong \frac{B \dot w' B}{\Ad_\d(B)}$ and $\frac{T}{\Ad_{\dot w \circ \d}(T)} \cong \frac{T}{\Ad_{\dot w' \circ \d}(T)}$ we construct can be regarded as the geometric lifting of the path of cyclic shifts $w=w_1 \xrightarrow{s_1} \cdots \xrightarrow{s_{n-1}} w_n=w'$. Such isomorphisms depend on the choice of such paths, and there seem to be no canonical choices of isomorphisms. 
\end{remark}

\begin{proof}
It suffices to consider the case where $w \xrightarrow{s} w'$. In this case, $\ell(w')=\ell(w)$. It is easy to see that if $s w>w$ and $w \ds(s)>w$, then $w'=w$. Without loss of generality, we may assume furthermore that $s w<w$. Set $w_1=s w$. Then $w_1 \d(s)=w'$. We have $$B \dot s B \times^B B \dot w_1 B \cong B \dot w B,$$ where $B$ acts on $B \dot s B \times B \dot w_1 B$ via $b \cdot (g_1, g_2)=(g_1 b \i, b g_2)$ and $B \dot s B \times^B B \dot w_1 B$ is the quotient space. 
Similarly, we have $$B \dot w_1 B \times^B B \d(\dot s) B \cong B \dot w' B.$$

We have a natural isomorphism $$B \dot s B \times B \dot w_1 B \cong B \dot w_1 B \times B \d(\dot s) B, \qquad (g_1, g_2) \mapsto (g_2, \d(g_1)).$$ By definition, $g_1 g_2$ and $g_2 \d(g_1)$ are in the same $\d$-conjugacy class of $G$. This isomorphism induces the desired isomorphism $\frac{B \dot w B}{\Ad_\d(B)} \cong \frac{B \dot w' B}{\Ad_\d(B)}$. 

Note that $\Ad(\dot s) \circ \Ad(\dot w) \circ \d=\Ad(\dot w') \circ \d \circ \Ad(\dot s)$ on $T$. The desired isomorphism  $\frac{T}{\Ad_{\dot w \circ \d}(T)} \cong  \frac{T}{\Ad_{\dot w' \circ \d}(T)}$ is induced from the conjugation action by $\dot s$. 
\end{proof}

\begin{corollary}
Let $w, w' \in W$ with $w \approx_\d w'$. Let $f: \frac{T}{\Ad_{\dot w \circ \d}(T)} \cong  \frac{T}{\Ad_{\dot w' \circ \d}(T)}$ be an isomorphism in Proposition \ref{prop:cyc}. Then we have $$(\pi_w)_! (p_w)^*=(\pi_{w'})_! (p_{w'})^* f: \text{Sh}(\frac{T}{\Ad_{\dot w \circ \d}(T)}) \to \text{Sh}(\frac{G}{\Ad_\d(G)}).$$ Here $\text{Sh}(-)$ denotes the derived category of complexes of sheaves. 
\end{corollary}

\section{Parabolic character sheaves}

\subsection{Partial conjugation on $W$}\label{sec:partial} Let $J \subset S$. Recall that $W_J$ is the subgroup of $W$ generated by the simple reflections in $J$. Let $W^J$ (resp. ${}^J W$) be the set of minimal coset representatives in $W/W_J$ (resp. $W_J \backslash W$). For $J, K \subset S$, we simply write ${}^J W^K$ for ${}^J W \cap W^K$. 

Consider the $\d$-conjugation action of $W_J$ on $W$ defined by $w \cdot_\d w'=w w' \d(w) \i$ for $w \in W_J$ and $w' \in W$. For $w \in W$, we write $\Ad(w) \d(J)=J$ if for any simple reflection $s \in J$, there exists a simple reflection $s' \in J$ such that $w \d(s) w \i=s'$. In this case, $w \in {}^J W$ if and only if $w \in W^{\d(J)}$. 

For any $J \subset S$ and $w \in W$, we set $$I(J, w, \d)=\max\{K \subset J; \Ad(w) \d(K)=K\}.$$ 

Now we recall B\'edard's description \cite{Be} of ${}^J W$. 

Let $\CT(J, \d)$ be the set of sequences $(J_n, w_n)_{n \ge 0}$ with $J_n \subset J$ and $w_n \in W$ such that
\begin{enumerate}
    \item $J_0=J$; 
    
    \item $J_n=J_{n-1} \cap \Ad(w_{n-1}) (\d(J_{n-1}))$ for $n \ge 1$; 
    
    \item $w_n \in {}^{J_n} W^{\d(J_n)}$ for $n \ge 0$; 
    
    \item $w_n \in W_{J_{n}} w_{n-1} W_{\d(J_{n-1})}$ for $n \ge 1$. 
\end{enumerate}

Then for any sequence $(J_n, w_n)_{n \ge 0}$, we have that $w_m=w_{m+1}=\cdots$ and $J_m=J_{m+1}=\cdots$ for $m \ge 0$. By \cite[Proposition 2.5]{Lu1}, the assignment $(J_n, w_n)_{n \ge 0} \mapsto w_m$ for $m \gg 0$ defines a bijection $\CT(J, \d) \to {}^J W$.
Moreover, $w_n \in w W_{\d(J_n)}$ for all $n \ge 0$ and $I(J, w, \d)=J_m$ for $m \gg 0$. 

We have the following description of $W_{I(J, w, \d)}$. 

\begin{lemma}\label{lem:J-inf}
Let $w \in {}^J W$. Then $W_{I(J, w, \d)}=\bigcap_{n \in \BZ} (\Ad(w) \circ \d)^n (W_J)$. 
\end{lemma}

\begin{proof}
By definition, $\Ad(w) \circ \d(W_{I(J, w, \d)})=W_{I(J, w, \d)} \subset W_J$. 

Let $(J_n, w_n)_{n \ge 0}$ be the sequence in $\CT(J, \d)$ corresponding to $w$. Now we prove that $\cap_{n \in \BZ} (\Ad(w) \circ \d)^n (W_J) \subset W_{J_n}$ for all $n$. 

By definition $\cap_{n \in \BZ} (\Ad(w) \circ \d)^n (W_J) \subset W_J=W_{J_0}$. Assume that we have proved $\cap_{n \in \BZ} (\Ad(w) \circ \d)^n (W_J) \subset W_{J_i}$ for some $i$. By definition, $w=w_i \d(a)$ for some $a \in W_{J_i}$. Hence 
\begin{align*}
    & \cap_{n \in \BZ} (\Ad(w) \circ \d)^n (W_J) \\ &=\bigl(\cap_{n \in \BZ} (\Ad(w) \circ \d)^n (W_J) \bigr) \cap \Ad(w) \circ \d \bigl(\cap_{n \in \BZ} (\Ad(w) \circ \d)^n (W_J) \bigr) \\ & \subset W_{J_i} \cap \Ad(w) \circ \d(W_{J_i})=W_{J_i} \cap \Ad(w_i) \circ \d(W_{J_i}) \\ &=W_{J_{i+1}}.
\end{align*}

This finishes the proof. 
\end{proof}

We have the following classification of the $\Ad_\d(W_J)$-orbits on $W$. 

\begin{proposition}\label{prop:w} Let $J \subset S$. Then 

(1) $W=\bigsqcup_{w \in {}^J W} W_J \cdot_\d (W_{I(J, w, \d)} w)$;

(2) For any $w \in {}^J W$, the embedding $W_{I(J, w, \d)} \to W$, $u \mapsto u w$ induces the bijection between the quotient stacks $$\frac{W_{I(J, w, \d)}}{\Ad_{w \circ \d}(W_{I(J, w, \d)})} \cong \frac{W_J \cdot_\d (W_{I(J, w, \d)} w)}{\Ad_\d(W_J)}.$$
\end{proposition}

\begin{proof}
Part (1) and the bijection of orbits in part (2) were proved in \cite[Corollary 2.6]{He07}. It remains to show that the bijection on the orbits also gives an isomorphism between the isotropy groups. In other words, it remains to show that if $(a, b) \in W_J \times W_{I(J, w, \d)} w$ with $a b \d(a) \i \in W_{I(J, w, \d)} w$, then $a \in W_{I(J, w, \d)}$. 

Let $(J_n, w_n)_{n \ge 0}$ be the element in $\CT(J, \d)$ corresponding to $w$. We argue by induction that $a \in W_{J_n}$ for all $n$. 

By definition, $a \in W_J=W_{J_0}$. Assume that $a \in W_{J_i}$ for some $i$. Then $W_{I(J, w, \d)} w \subset w_i W_{\d(J_i)}$ and $a b \d(a) \i \in a w_i W_{\d(J_i)}$. Thus $$a \in W_{J_i} \cap w_i W_{\d(J_i)} w_i \i=W_{J_{i+1}}.$$

Hence $a \in W_{J_n}$ for all $n \ge 0$. In particular, $a \in W_{I(J, w, \d)}$. 
\end{proof}

\subsection{Lusztig's variety $Z_{J, \d}$}\label{sec:zjw}
For any $J \subset S$, let $P_J \supset B$ be the standard parabolic subgroup of type $J$. Let $\CP_J=G/P_J$ be the partial flag variety. We may identify $\CP_J$ with the set of parabolic subgroups of $G$ that are conjugate to $P_J$. For any parabolic subgroup $P$ of $G$, we denote by $U_P$ its unipotent radical. For $g \in G$ and $H \subset G$, we simply write ${}^g H$ for $g H g \i$. Following \cite{Lu1}, we set $$Z_{J, \d}=\{(P, P', g \d(U_P)); P \in \CP_J, P' \in \CP_{\d(J)}, g \in G, {}^g \d(P)=P'\}.$$ Define the action of $G \times G$ on $Z_{J, \d}$ by $$(g_1, g_2) \cdot (P, P', g \d(U_P))=({}^{g_2} P, {}^{g_1} P', g_1 g \d(g_2) \i \, \d(U_{{}^{g_2} P})).$$ Then $G \times G$ acts transitively on $Z_{J, \d}$. Let $h_{J, \d}=(P_J, P_{\d(J)}, U_{P_{\d(J)}})$ be the base point. Then the isotropy group of $h_{J, \d}$ is $\{(\d(l) u', l u); l \in L_J, u \in U_{P_J}, u' \in U_{P_{\d(J)}}\}$. The map $G \times G \to G, (g, g') \mapsto (g') \i g$ induces an isomorphism of stacks $$\frac{Z_{J, \d}}{\D(G)} \cong \CY_{J, \d}.$$
Here $\frac{Z_{J, \d}}{\D(G)}$ is the quotient stack for the diagonal $G$-action on $Z_{J, \d}$ and $\CY_{J, \d}=\frac{U_{P_J} \backslash G/U_{P_{\d(J)}}}{\Ad_\d(L_J)}$ is the quotient stack for the $\d$-conjugation action of $L_J$ on $U_{P_J} \backslash G/U_{P_{\d(J)}}$.

Now we recall the decomposition of $Z_{J, \d}$ into the $G$-stable pieces, introduced by Lusztig in \cite{Lu1}. For $w \in {}^J W$, we set $$Z_{J, \d; w}=G_\D \cdot (B \dot w, B) h_{J, \d}.$$ The following result is established by Lusztig in \cite{Lu1}. 

\begin{proposition}\label{prop:g-stable} Let $J \subset S$. Then

(1) $Z_{J, \d}=\bigsqcup_{w \in {}^J W} Z_{J, \d; w}$ is a decomposition into smooth, locally closed subvarieties. 

(2) For any $w \in {}^J W$, there is a canonical map $$\pi_{J, \d; w}: \frac{Z_{J, \d; w}}{\D(G)} \to \frac{L_{I(J, w, \d)}}{\Ad_{\dot w \circ \d}(L_{I(J, w, \d)})}$$ which is an iterated gerbe for unipotent groups.
\end{proposition}

Under the isomorphism $\frac{Z_{J, \d}}{\D(G)} \cong \CY_{J, \d}$, we may reformulate the above decomposition as follows. For any $w \in W$, let $\CY_{J, \d;w}$ be the image of $B \dot w B$ in $\CY_{J, \d}$. Then for any $w \in {}^J W$, $\frac{Z_{J, \d; w}}{\D(G)} \cong \CY_{J, \d;w}$ and $\CY_{J, \d}=\sqcup_{w \in {}^J W} \CY_{J, \d;w}$. 

By \cite[Proposition 1.10 (1)]{He07a}, $$Z_{J, \d; w} \cong G \times^{P_{I(J, w, \d)}} (P_{I(J, w, \d)} \dot w, P_{I(J, w, \d)}) \cdot h_{J, \d},$$ where $P_{I(J, w, \d)}$ acts on $G \times (P_{I(J, w, \d)} \dot w, P_{I(J, w, \d)}) \cdot h_{J, \d}$ by $p \cdot (g, z)=(g p \i, (p, p) \cdot z)$ and $G \times^{P_{I(J, w, \d)}} (P_{I(J, w, \d)} \dot w, P_{I(J, w, \d)}) \cdot h_{J, \d}$ is the quotient space. We may reformulate it as follows. 

\begin{proposition}\label{prop:cyw}
Let $w \in {}^J W$. The embedding $P_{I(J, w, \d)} \dot w P_{\d(I(J, w, \d)} \to G$ induces an isomorphism $$\frac{U_{P_J} \backslash P_{I(J, w, \d)} \dot w P_{\d(I(J, w, \d)}/U_{P_{\d(J)}}}{\Ad_\d(L_J \cap P_{I(J, w, \d)})} \cong \CY_{J, \d;w}.$$
\end{proposition}

\subsection{Parabolic character sheaves}
We follow \cite{Lu1}. For any $w \in W$, we consider the following diagram 
\[
\xymatrix{
\frac{T}{\Ad_{\dot w \circ \d}(T)} & \frac{B \dot w B}{\Ad_\d(B)} \ar[l]_-{p_{w}} \ar[r]^-{\pi_{J, w}} & \CY_{J, \d}.
}
\]

A parabolic character sheaf on $\CY_{J, \d}$ is a simple perverse sheaf that is a composition factor of ${}^p H^i((\pi_{J, w})_! p_w^* \CL)$ for some $w \in W$, $i \in \BZ$ and $\CL \in \Sh(\frac{T}{\Ad_{\dot w \circ \d}(T)})$. 

On the other hand, by Proposition \ref{prop:g-stable}, for any $w \in {}^J W$, the map $\pi_{J, w}$ induces an equivalence of categories $$\pi_{J, \d; w}^*: \text{Sh}(\frac{L_{I(J, w, \d)}}{\Ad_{\dot w \circ \d}(L_{I(J, w, \d)})}) \cong \text{Sh}(\frac{Z_{J, \d; w}}{\D(G)})=\Sh(\CY_{J, \d;w}).$$

Note that $L_{I(J, w, \d)}$ is a connected reductive group. The character sheaves on $\frac{L_{I(J, w, \d)}}{\Ad_{\dot w \circ \d}(L_{I(J, w, \d)})}$ is defined by Lusztig in \cite{Lu2}. 

It is proved by Lusztig in \cite{Lu1} that 

\begin{theorem} Let $J \subset S$. Then 

(1) Any parabolic character sheaf on $\CY_{J, \d}$ is the intermediate extension of $\pi^*_{J, \d; w}(A)$ for a unique $w \in {}^J W$ and a character sheaf $A$ on $\frac{L_{I(J, w, \d)}}{\Ad_{\dot w \circ \d}(L_{I(J, w, \d)})}$. 

(2) If $B$ is a parabolic character sheaf on $\CY_{J, \d}$ and $w \in {}^J W$, then any composition factor of ${}^p H^i(B \mid_{\CY_{J, \d;w}})$ is of the form $\pi^*_{J, \d; w}(A)$ for some character sheaf $A$ on $\frac{L_{I(J, w, \d)}}{\Ad_{\dot w \circ \d}(L_{I(J, w, \d)})}$. 
\end{theorem}

\subsection{Partial conjugation graph} Let $J \subset S$. We consider the $\Ad_\d(W_J)$-orbits on $W$. The {\it $(J, \d)$-conjugacy graph}, by definition, is the direct graph with vertices in $W$ and the edges are of the form $w \xrightarrow{s}_\d w'$ for $w, w' \in W$, $s \in J$. 

We denote by $\to_{J, \d}$ the pre-order relation induced by $\xrightarrow{s}_\d$ for $s \in J$. We write $w \approx_{J, \d} w'$ if $w \to_{J, \d} w'$ and $w' \to_{J, \d} w$. We call the equivalece class $\approx_{J, \d}$ the $(J, \d)$-cyclic shift classes on $W$. 

We have the following result. 

\begin{proposition}\cite[Proposition 3.4]{He07}
For any $w \in W$, there exists $w' \in {}^J W$ and $u \in W_{I(J, w', \d)}$ such that $w \to_{J, \d} u w'$. 

Moreover, if $w$ is of minimal length in $W_J \cdot_\d w$, then $u$ is of minimal length in $W_{I(J, w', \d)} \cdot_{\d'} u$ and $w \approx_{J, \d} u w'$, where $\d'=\Ad(w') \circ  \d$. 
\end{proposition}

\begin{remark}
By Proposition \ref{prop:w} (1), $w'$ is uniquely determined by $w$. However, $u$ is not unique in general. 
\end{remark}

\subsection{Partial order}\label{sec:partial}
For any $w \in W$ and $w' \in {}^J W$, we write $w' \le_{J, \d} w$ if there exists $u \in W_J$ such that $u w' \d(u) \i \le w$. By \cite[Corollary 4.6]{He07}, 

(a) {\it The restriction of $\le_{J, \d}$ to ${}^J W$ gives a partial order on ${}^J W$.}

It is easy to see that for $w, w' \in {}^J W$, $w' \le w$ implies that $w' \le_{J, \d} w$ and $w' \le_{J, \d} w$ implies that $\ell(w') \le \ell(w)$. However, the converse directions do not hold in general. 

\begin{example}
Let $W=S_4$ and $J=\{3\}$. The simple reflections of $W$ are $s_1, s_2, s_3$. We simply write $s_{a b c \cdots}$ instead of $s_a s_b s_c \cdots$. In Figure \ref{fig1.1}, we draw the Hasse diagram of ${}^J W$, with respect to the usual Bruhat order and the partial order $\le_{J, \id}$ (the extra relation is in dotted line). 
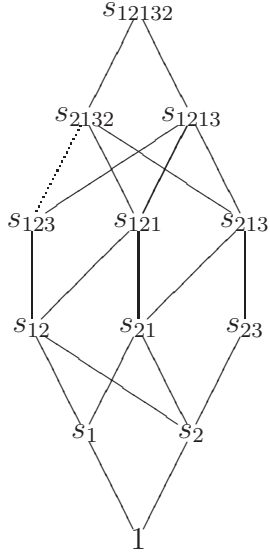
\begin{figure}[h]
  \caption{Hasse diagram for the partial orders on ${}^J W$}
\[
\begin{xy}
0;<0.7cm,0pt>:
(0, 8)*{s_{12132}}="12132";
(-1,6)*{s_{2132}}="2132";
(1,6)*{s_{1213}}="1213";
(-2,4)*{s_{123}}="123";
(0,4)*{s_{121}}="121";
(2,4)*{s_{213}}="213";
(-2,2)*{s_{12}}="12";
(0,2)*{s_{21}}="21";
(2,2)*{s_{23}}="23";
(-1,0)*{s_1}="1";
(1,0)*{s_2}="2";
(0,-2)*{1}="0";
"1" **@{-};
"12" **@{-};
"123" **@{-};
"2132" **@{.};
"12132" **@{-};
"1213" **@{-};
"121" **@{-};
"12" **@{-};
"2" **@{-};
"0" **@{-};
"1";
"21" **@{-};
"121" **@{-};
"2132" **@{-};
"213" **@{-};
"1213" **@{-};
"121" **@{-};
"2";
"21" **@{-};
"213" **@{-};
"23" **@{-};
"2" **@{-};
"123";
"1213" **@{-};
\end{xy}
\]
\label{fig1.1}
\end{figure}
\end{example}

By \cite[Proposition 5.8]{He07}, we have 

(a) {\it 
Let $w \in W$. Then $\overline{\CY_{J, \d;w}}=\bigsqcup_{w' \in {}^J W; w' \le_{J, \d} w} \CY_{J, \d;w'}$.}

Since $\CY_{J, \d;w}$ is irreducible, there exists a unique geometric piece $\CY_{J, \d;w'}$ which is dense in $\overline{\CY_{J, \d;w}}$. Therefore we have

(b) {\it 
For any $w \in W$, the set $\{w' \in W^J; w' \le_{J, \d} w\}$ contains a unique maximal element with respect to $\le_{J, \d}$.}

We also would like to point out the special case of (a), which will be used in \cite{LNY}. 

(c) {\it Let $w \in W$. If $W_J \cdot_\d w \cap {}^J W=\emptyset$ or $w$ is not of minimal length in $W_J \cdot_\d w$, then $\CY_{J, \d;w} \subset \bigsqcup_{w' \in {}^J W; \ell(w')<\ell(w)} \CY_{J, \d;w'}$.}

\section{Cyclic shifts of pieces}

\subsection{Combinatorial pieces} %By the similar argument as in Proposition \ref{prop:cyc}, we have $\CY_{J, \d;w}=\CY_{J, \d;w'}$ if $w \approx_{J, \d} w'$. On the other hand, as we see in \S \ref{sec:zjw}, we also consider the variety $P_{I(J, w, \d)} \dot w$ together with the action of $L_{I(J, w, \d)}$, and the sheaves on it. Motivated by the geometric pieces, we introduce the {\it combinatorial pieces} as follows. 

A {\it $\d$-combinatorial piece} is a pair $(w, K)$, where $w \in W$ and $K \subset S$ with $w \in {}^K W$ and $\Ad(w) \d(K)=K$ (and hence $w \in W^{\d(K)}$). To each $\d$-combinatorial piece $(w, K)$, we associate a subset $W_K w$ of $W$. In particular, if $K=\emptyset$, then we naturally identify the $\d$-combinatorial piece $(w, \emptyset)$ with the element $w$ of $W$. In this way, we identify $W$ as a subset of the set of $\d$-combinatorial pieces. 

We say that two $\d$-combinatorial pieces $(w, K)$ and $(w', K')$ are $\d$-conjugated by an element $x \in W$ if $w'=x \i w \d(x)$, and $\Ad(x) \i(K)=K'$ (and hence $x \in W^K$). In this case, $x \i (W_K w) \d(x)=W_{K'} w'$. 

It is worth pointing out that the whole Weyl group $W$ does not act on the set of $\d$-combinatorial pieces. The element $x$ acts on $(w, K)$ only if $\Ad(x) \i (K) \subset S$. 

\subsection{Cyclic shifts of pieces}\label{sec:def} The definition of cyclic shifts on $W$ in \S \ref{sec:cyc} does not generalize to the set of $\d$-combinatorial pieces. 

However, there is an equivalent definition of cyclic shifts on $W$, due to Brou\'e and Michel in \cite{BM}. The definition is as follows. Let $w, w' \in W$. Then $w \approx_{\d} w'$ if there exists a sequence of elements $w=w_1, \ldots, w_n=w'$ and elements $x_i, y_i \in W$ such that $w_i=x_i y_i$, $w_{i+1}=y_i \d(x_i)$ and $\ell(w_i)=\ell(x_i)+\ell(y_i)=\ell(w_{i+1})$ for all $1 \le i \le n-1$. See \cite[Exercise 3.7]{GP00}. 

Now we define cyclic shifts of the combinatorial pieces. 

Let $(w, K)$ and $(w', K')$ be $\d$-combinatorial pieces. We write $(w, K) \stackrel{x}{\approx}_\d (w', K')$ if $(w, K)$ and $(w', K')$ are $\d$-conjugated by $x$, and $\ell(w)=\ell(x)+\ell(x \i w)=\ell(w')$. Let $\approx_\d$ be the equivalence relation on the set of $\d$-combinatorial pieces generated by $\stackrel{x}{\approx}_\d$ for $x \in W$. We call it the {\it cyclic shift relation of the $\d$-combinatorial pieces}. When restricting to $W$, this definition coincides with the definition of Brou\'e and Michel. 

Similarly, for any $J \subset S$, let $\approx_{J, \d}$ be the equivalence relation on the set of $\d$-combinatorial pieces generated by $\stackrel{x}{\approx}_\d$ for $x \in W_J$. In geometric applications, we usually consider the $\approx_{J, \d}$ on the set of $\d$-combinatorial pieces $(w, K)$ with the additional assumption that $K \subset J$. 

\begin{proposition}\label{prop:conj}
Let $(w, K)$ and $(w', K')$ be $\d$-combinatorial pieces with $K, K' \subset J$. Suppose that $(w, K) \approx_{J, \d} (w', K')$. Then there exists isomorphisms $\frac{U_{P_J} \backslash P_K \dot w P_{\d(K)} /U_{P_{\d(J)}}}{\Ad_{\dot w \circ \d}(L_J \cap P_K)} \cong \frac{U_{P_J} \backslash P_{K'} \dot w' P_{\d(K')}/U_{P_{\d(J)}}}{\Ad_{\dot w' \circ \d}(L_J \cap P_{K'})}$ and $\frac{L_K}{\Ad_{\dot w \circ \d}(L_K)} \cong \frac{L_{K'}}{\Ad_{\dot w' \circ \d}(L_{K'})}$ such that the following diagram commutes
\[
\xymatrix{
\frac{L_K}{\Ad_{\dot w \circ \d}(L_K)} \ar@{-->}[d]_\cong & \frac{U_{P_J} \backslash P_K \dot w P_{\d(K)} /U_{P_{\d(J)}}}{\Ad_{\dot w \circ \d}(L_J \cap P_K)} \ar[l] \ar[r] \ar@{-->}[d]^\cong & \CY_{J, \d} \ar@{=}[d] \\
\frac{L_{K'}}{\Ad_{\dot w' \circ \d}(L_{K'})} & \frac{U_{P_J} \backslash P_{K'} \dot w' P_{\d(K')}/U_{P_{\d(J)}}}{\Ad_{\dot w' \circ \d}(L_J \cap P_{K'})} \ar[l] \ar[r] & \CY_{J, \d}.
}
\]
\end{proposition}

\begin{proof}
It suffices to consider the case where $(w, K) \stackrel{x}{\approx}_\d (w', K')$ for some $x \in W_J$. Set $y=x \i w$. Then $w'=y \d(x)$. We have $$P_K \dot w P_{\d(K)}=U_{P_K} \dot w P_{\d(K)} \cong (U_{P_K} \cap \dot w U^- \dot w \i) \times P_{\d(K)}.$$ Moreover, $U_{P_K} \cap \dot w U^- \dot w \i=U \cap \dot w U^- \dot w \i$. Note that $\Ad(y) \d(K)=\Ad(x \i w) \d(K)=\Ad(x) \i(K)=K'$. Similarly, we have 
\begin{gather*}
    P_K \dot x P_{K'} \cong (U \cap \dot x U^- \dot x \i) \times P_{K'};\\
    P_{K'} \dot y P_{\d(K)} \cong (U \cap \dot y U^- \dot y \i) \times P_{\d(K)}.
\end{gather*}

Since $\ell(w)=\ell(x)+\ell(y)$, we have $U \cap \dot w U^- \dot w \i \cong (U \cap \dot x U^- \dot x \i) \times (U \cap \dot y U^- \dot y \i)$. Hence 
\begin{align*}
    P_K \dot w P_{\d(K)} & \cong U \cap \dot w U^- \dot w \i \times P_{\d(K)} \\ & \cong (U \cap \dot x U^- \dot x \i) \times (U \cap \dot y U^- \dot y \i) \times P_{\d(K)} \\ & \cong (U \cap \dot x U^- \dot x \i) \times P_{K'} \dot y P_{\d(K)} \\ & \cong \bigl((U \cap \dot x U^- \dot x \i) \times P_{K'}\bigr) \times^{P_{K'}} P_{K'} \dot y P_{\d(K)} \\ & \cong P_K \dot x P_{K'} \times^{P_{K'}} P_{K'} \dot y P_{\d(K)}. 
\end{align*}

The map $(g_1, g_2) \mapsto (g_2, \d(g_1))$ gives a natural isomorphism $$f: \frac{P_K \dot x P_{K'} \times^{P_{K'}} P_{K'} \dot y P_{\d(K)}}{\Ad_{\d}(P_K)} \cong \frac{P_{K'} \dot y P_{\d(K)} \times^{P_{\d(K)}} P_{\d(K)} \d(\dot x) P_{\d(K')}}{\Ad_{\d}(P_{K'})}.$$ 

Since $K, K' \subset J$ and $x \in W_J$, we have $P_K \dot x P_{K'} \subset P_K$ and hence the conjugation action of $P_K \dot x P_{K'}$ normalizes $U_{P_J}$. Thus $f$ induces the desired isomorphism $\frac{U_{P_J} \backslash P_K \dot w P_{\d(K)} /U_{P_{\d(J)}}}{\Ad_\d(L_J \cap P_K)} \cong \frac{U_{P_J} \backslash P_{K'} \dot w' P_{\d(K')}/U_{P_{\d(J)}}}{\Ad_\d(L_J \cap P_{K'})}$. 

Note that the representative of $w'$ in $N_G(T)$ is unique up to right multiplication by $T$. As we consider the morphisms on the quotient stacks, we may assume furthermore that $\dot w'$ is chosen so that $\dot w'=\dot x \i \dot w \d(\dot x)$. Then $\Ad(\dot x) \i \circ \Ad(\dot w) \circ \d=\Ad(\dot w') \circ \d \circ \Ad(\dot x) \i$ on $L_K$. The desired isomorphism  $\frac{L_K}{\Ad_{\dot w \circ \d}(L_K)} \cong \frac{L_{K'}}{\Ad_{\dot w' \circ \d}(L_{K'})}$ is induced from the conjugation action by $\dot x \i$. 
\end{proof}

\begin{theorem}\label{thm:cyc}
Let $J \subset S$ and $(w, K)$ be a $\d$-combinatorial piece. Suppose that $K \subset J$ and $w$ is of minimal length in $W_J \cdot_\d w$. Then there exists unique $w' \in {}^J W$, $x \in W_J \cap W^{I(J, w', \d)}$, $u \in W_{I(J, w', \d)}$ such that $x \i w \d(x)=u w'$, $\Ad(x) \i(K) \subset J$ and $(w, K) \approx_{J, \d} (u w', \Ad(x) \i(K))$. 
\end{theorem}

\begin{proof}
By Proposition \ref{prop:w} (1), there exists a unique $w' \in {}^J W$ such that $w \in W_J \cdot_\d W_{I(J, w', \d)} w'$. By Proposition \ref{prop:w} (2), there exists a unique $x \in W_J \cap W^{I(J, w', \d)}$ such that $x \i w \d(x) \in W_{I(J, w', \d)} w'$. Let $u \in W_{I(J, w', \d)}$ with $x \i w \d(x)=u w'$. 

By Lemma \ref{lem:J-inf}, \begin{align*} W_{I(J, w', \d)} &=\cap_{n \in \BZ} (\Ad(w') \circ \d)^n (W_J)=\cap_{n \in \BZ} (\Ad(u w') \circ \d)^n (W_J) \\ &=\cap_{n \in \BZ} (\Ad(x) \i \Ad(w) \circ \d \circ \Ad(x))^n (W_J) \\ &=\cap_{n \in \BZ} \Ad(x) \i (\Ad(w) \circ \d)^n \Ad(x) (W_J) \\ &=\Ad(x) \i \bigl(\cap_{n \in \BZ} (\Ad(w) \circ \d)^n(W_J) \bigr) \\ & \supset \Ad(x) \i(W_K).
\end{align*}
The second equality uses the fact that $(\Ad(u w') \circ \d)^n=(\Ad(w') \circ \d)^n \Ad(u')$ for some $u' \in W_{I(J, w', \d)}$. 

Therefore $x \i$ sends any simple root in $K$ to a root spanned by $I(J, w', \d)$. Since $x \in W^{I(J, w', \d)}$, we have $\Ad(x) \i(K) \subset I(J, w', \d)$. Let $K'=\Ad(x) \i(K)$. Then \begin{align*} \Ad(u w') \d(K') &=\Ad(x) \i \Ad(w) \d \Ad(x)(K')=\Ad(x) \i \Ad(w) \d(K) \\ &=\Ad(x) \i(K)=K'.\end{align*} Thus $(u w', K')$ is a $\d$-combinatorial piece. It remains to prove that $(w, K) \approx_{J, \d} (u w', K')$. 

We associate a quadruple $(J_n, w_n, x_n, y_n)_{n \ge 0}$ to $w$. Here $J_n \subset J$, $w_n \in {}^{J_n} W^{\d(J_n)}$, $x_n \in W_{J_n} \cap W^{J_{n+1}}$ and $y_n \in W_{J_{n+1}}$ for all $n$. The quadruple is constructed as follows. 

Set $J_0=J$. Suppose that $n \ge 0$ and $(J_m, x_m)$ are defined for $0 \le m<n$. We set $w'_n=(x_0 \cdots x_{n-1}) \i w \d(x_0 \cdots x_{n-1})$. We may write $w'_n$ as $w'_n=z_n w''_n$, where $w''_n \in {}^{J_n} W$ and $z_n \in W_{J_n}$. Set $w_n=\min(w''_n W_{\d(J_n)})$ and $J_{n+1}=J_n \cap \Ad(w_n) \d(J_n)$. We write $z_n$ as $z_n=x_n y_n$, where $x_n \in W_{J_n} \cap W^{J_{n+1}}$ and $y_n \in W_{J_{n+1}}$. The quadruple $(J_n, w_n, x_n, y_n)_{n \ge 0}$ is defined inductively. 

By definition, $(J_n, w_n)_{n \ge 0} \in \CT(J, \d)$. Then there exists $m$ such that $J_m=J_{m+1}=\cdots$, $w_m=w_{m+1}=\cdots$. We then have $x_m=x_{m+1}=\cdots=1$ and $y_m=y_{m+1}=\cdots$. Moreover, $J_m=I(J, w_m, \d)$. Set $x'=x_0 x_1 \cdots x_m \in W_{J} \cap W^{I(J, w_m, \d)}$. Then $(x') \i w x' \in W_{J_m} w_m=W_{I(J, w_m, \d)} w_m$. By Proposition \ref{prop:w}, we have $w_m=w'$ and $x'=x$. 

Moreover, we have $\ell(w'_n)=\ell(w''_n)+\ell(z_n)=\ell(w''_n)+\ell(x_n)+\ell(y_n)=\ell(y_n w''_n)+\ell(x_n)$. Thus 
\begin{align*} \ell(w'_{n+1}) &=\ell(x_n \i w'_n \d(x_n)) \le \ell(x_n \i w'_n)+\ell(x_n)=\ell(y_n w''_n)+\ell(x_n) \\ &=\ell(w'_n).
\end{align*}

In particular, we have $\ell(w)=\ell(w'_0) \ge \ell(w'_1) \ge \cdots \ge \ell(w'_m)=\ell(x \i w \d(x))$. By our assumption, $w$ is of minimal length in $W_J \cdot_\d w$. Hence $\ell(w)=\ell(w'_1)=\cdots=\ell(w'_m)=\ell(x \i w \d(x))$. Moreover, we have $w=w'_0 \approx_{J, \d} w'_1 \approx_{J, \d} \cdots \approx_{J, \d} w'_m=x \i w \d(x)$. 

Note that $\Ad(x_0) \i(W_K)=\Ad(x_1 \cdots x_m) W_{K'} \subset W_{J_1}$. Since $x_0 \in W^{J_1}$, we have $\Ad(x_0) \i(K) \subset J_1$. By the same argument, we may show that $\Ad(x_0 \cdots x_i) \i(K) \subset J_{i+1}$ for all $i \ge 0$. 

Hence we have $(w, K) \stackrel{x_0}{\approx}_\d (w'_1, \Ad(x_0) \i (K)) \stackrel{x_1}{\approx}_\d \cdots \stackrel{x_m}{\approx}_\d (u w', K')$. 
\end{proof}

\section{Applications}\label{sec:app}

\subsection{Left-right symmetry}
Recall that $$W=\bigsqcup_{w \in {}^J W} W_J \cdot_\d (W_{I(J, w, \d)} w).$$ By \cite[\S 2]{He07}, we have the following decomposition of $W$ indexed by $W^{\d(J)}$ instead of ${}^J W$: $$W=\bigsqcup_{w \in W^{\d(J)}} W_J \cdot_\d (W_{I(J, w, \d)} w).$$

Now we show that there exists a natural bijection between ${}^J W$ and $W^{\d(J)}$ so that the corresponding subsets of $W$ coincide. 

\begin{proposition}
Let $J \subset S$. Then there exists a unique bijection $\iota: W^{\d(J)} \to {}^J W$ such that $$\bigl(w, I(J, w, \d) \bigr) \approx_{J, \d} \bigl(\iota(w), I(J, \iota(w), \d) \bigr).$$ In particular, $W_J \cdot_\d (W_{I(J, w, \d)} w)=W_J \cdot_\d (W_{I(J, \iota(w), \d)} \iota(w))$. 
\end{proposition}

\begin{proof}
Let $w \in W^{\d(J)}$ and $K=I(J, w, \d)$. By definition, $(w, K)$ is a $\d$-combinatorial piece. For any $u \in W_J$, we have $$\ell(u w \d(u) \i) \ge \ell(w \d(u) \i)-\ell(u)=\ell(w)+\ell(\d(u))-\ell(u)=\ell(w).$$ Therefore $w$ is of minimal length in $W_J \cdot_\d w$. By Theorem \ref{thm:cyc}, there exists unique $w' \in {}^J W$, $x \in W_J \cap W^{I(J, w', \d)}$, $u \in W_{I(J, w', \d)}$ such that $x \i w \d(x)=u w'$, $\Ad(x) \i(K) \subset J$ and $(w, K) \approx_{J, \d} (u w', \Ad(x) \i(K))$. 

Set $K'=\Ad(x) \i(K)$. By Lemma \ref{lem:J-inf}, $W_{K}=\cap_{n \in \BZ} (\Ad(w) \circ \d)^n (W_J)$. Hence \begin{align*} W_{K'} &=\Ad(x) \i(W_{K})=\cap_{n \in \BZ} \Ad(x) \i (\Ad(w) \circ \d)^n (W_J) \\ &=\cap_{n \in \BZ} (\Ad(u w') \circ \d)^n \Ad(x) \i (W_J)\\ &=\cap_{n \in \BZ} (\Ad(u w') \circ \d)^n (W_J) \\ &=\cap_{n \in \BZ} (\Ad(w') \circ \d)^n (W_J) \\ &=W_{I(J, w', \d)}.
\end{align*}
The second to last equality uses the fact that $(\Ad(u w') \circ \d)^n=(\Ad(w') \circ \d)^n \Ad(u')$ for some $u' \in W_{I(J, w', \d)}$. 

Therefore $K'=I(J, w', \d)$. Hence $x u w'=w \d(x) \in W^{\d(I(J, w', \d)}$. This implies that $u=1$. So $(w, K) \approx_{J, \d} (w', K')$. By the definition of cyclic shifts of combinatorial pieces, we also have $W_J \cdot_\d (W_{K} w)=W_J \cdot_\d (W_{K'} w')$. 
\end{proof}

Let $w \in W^{\d(J)}$ and $w'=\iota(w) \in {}^J W$. Since $w \approx_{J, \d} w'$, we have $\CY_{J, \d;w}=\CY_{J, \d;w'}$. Moreover, by Proposition \ref{prop:conj}, we have the following commutative diagram
\[
\xymatrix{
\Sh(\frac{L_{I(J, w, \d)}}{\Ad_{\dot w \circ \d}(L_{I(J, w, \d)})}) \ar[d]_{\Ad(\dot x) \i} \ar[r]^-{\cong} & \Sh(\CY_{J, \d;w}) \ar@{=}[d] \\
\Sh(\frac{L_{I(J, w', \d)}}{\Ad_{\dot w' \circ \d}(L_{I(J, w', \d)})}) \ar[r]^-{\cong} & \Sh(\CY_{J, \d;w'}),
}
\]
where $x$ is the unique element in $W_J \cap W^{I(J, w', \d)}$ such that $w'=x \i w \d(x)$. 

Now we provide a nontrivial example of the map $\iota$. 

\begin{example}
Let $W=S_5$, $J=\{1, 3\}$ and $\d$ is the identity map. Then one may check that $\iota(s_{1 2 1 3 2 4})=s_{2 1 3 2 4 3}$. In particular, the map $\iota$ is different from taking inverse, and different from the one-step operation $x y \mapsto y x$ for $x \in W_J$ and $y \in {}^J W$. 
\end{example}

\subsection{The functors $f^J_{J'}$}
We follow \cite[\S 6]{Lu1}. Let $J \subset J' \subset S$. Let $\pi: \CP_J \to \CP_{J'}$ be the projection map. Set $$Z_{J, J', \d}=\{(P, P', g \d(U_{\pi(P)})); P \in \CP_J, P' \in \CP_{\d(J)}, {}^g \d(P)=P'\}.$$ Define the action of $G \times G$ on $Z_{J, J', \d}$ by $(g_1, g_2) \cdot (P, P', g U_{\pi(P)})=({}^{g_2} P, {}^{g_1} P', g_1 g \d(g_2) \i \d(U_{\pi({}^{g_2} P)}))$. Then $G \times G$ acts transitively on $Z_{J, J', \d}$. Let $h_{J, J', \d}=(P_J, P_{\d(J)}, U_{P_{\d(J')}})$ be the base point. Then we may identify $Z_{J, J', \d}$ with $(G \times G)/(U_{P_{J'}} \times U_{P_{\d(J')}}) \Ad_\d(L_{J'} \cap P_J)$, where $L_{J'} \cap P_J$ is a standard parabolic subgroup of $L_{J'}$. 

We consider the following diagram 
\[
\xymatrix{Z_{J, \d} & Z_{J, J', \d} \ar[l]_-c \ar[r]^-d & Z_{J', \d},}
\]
where 
\begin{gather*} c(P, P', g \d(U_{\pi(P)}))=(P, P', g \d(U_{P})), \\ d(P, P', g \d(U_{\pi(P)}))=(\pi(P), \pi(P'), g \d(U_{\pi(P)})).
\end{gather*}

It is easy to see that $c$ is a locally trivial fibration with fibers isomorphic to an affine space $U_{P_J}/U_{P_{J'}}$ and $d$ is a locally trivial fibration with fibers isomorphic to the flag variety $L_{J'}/(L_{J'} \cap P_J)$. Set \begin{align*} f^J_{J'}=d_! c^*: \Sh(\frac{Z_{J, \d}}{\D G}) \to \Sh(\frac{Z_{J', \d}}{\D G}). %\\ \quad e_J^{J'}=c_! d^*: \Sh(\frac{Z_{J', \d}}{\D G}) \to \Sh(\frac{Z_{J, \d}}{\D G}).
\end{align*}

We may reformulate the functors as follows. Consider the following diagram of stacks 
\[
\xymatrix{
\CY_{J, \d} & \CY_{J, J', \d} \ar[l]_-q \ar[r]^-p & \CY_{J', \d} %\frac{U_{P_J} \backslash G/U_{P_{\d(J)}}}{\Ad_\d(L_J)} & \frac{U_{P_{J'}} \backslash G/U_{P_{\d(J')}}}{\Ad_\d(L_{J'} \cap P_J)} \ar[l]_-q \ar[r]^-p & \frac{U_{P_{J'}} \backslash G/U_{P_{\d(J')}}}{\Ad_\d(L_{J'})},
}
\] where $\CY_{J, J', \d}=\frac{U_{P_{J'}} \backslash G/U_{P_{\d(J')}}}{\Ad_\d(L_{J'} \cap P_J)}$ and $q$ and $p$ are natural projection maps. Under the isomorphism $\frac{Z_{J, \d}}{\D(G)} \cong \CY_{J, \d}$ and $\frac{Z_{J', \d}}{\D(G)} \cong \CY_{J', \d}$, we may rewrite the functor $f^J_{J'}$ as 
\begin{align*} f^J_{J'}=p_! q^*: \Sh(\CY_{J, \d}) \to \Sh(\CY_{J', \d}). % \\ \quad e_J^{J'}=q_! p^*: \Sh(\CY_{J', \d}) \to \Sh(\CY_{J, \d}).
\end{align*}

Let $w \in {}^J W$ and $w' \in {}^{J'} W$. Let $i_{J, \d; w}: \CY_{J, \d;w} \to \CY_{J, \d}$ and $i_{J', \d; w'}: \CY_{J', \d; w'} \to \CY_{J', \d}$ be the embeddings. We define 
\begin{align*} f^{J, w}_{J', w'}=i_{J', \d; w'}^* \, p_! q^* (i_{J, \d; w})_!: \Sh(\CY_{J, \d;w}) \to \Sh(\CY_{J', \d;w'}). %\\ \quad e_{J, w}^{J', w'}=i_{J, \d; w}^* \, q_! p^* (i_{J', \d; w'})_!: \Sh(\CY_{J', \d;w'}) \to \Sh(\CY_{J, \d;w}).
\end{align*}

By definition, for any $w \in {}^J W$, $p(q \i(\CY_{J, \d; w}))=\CY_{J', \d; w}$. Hence by \S\ref{sec:partial} (c), we have 

(a) {\it Let $w \in {}^J W$ and $A$ be a parabolic character sheaf on $\CY_{J, \d}$ with support in $\overline{\CY_{J, \d; w}}$. Assume that $W_{J'} \cdot_\d w \cap {}^{J'} W=\emptyset$ or $w$ is not of minimal length in $W_{J'} \cdot_\d w$. Then the support of $f^J_{J'}(A)$ is contained in $\sqcup_{w' \in {}^{J'} W; \ell(w')<\ell(w)} \CY_{J', \d; w'}$. 
}
%For any $w \in W^J$, we define $Z_{J, J', w}=c \i(Z_{J, \d; w})=G_{\D} (B \dot w, B) h_{J, J'}$. Then $Z_{J, J', \d}=\sqcup_{w \in W^J} Z_{J, J', w}$. 
%For any $w \in W^J$, we have 
%\begin{align*} d \circ c \i(\overline{Z_{J, \d; w}}) &=d(\overline{c \i(Z_{J, \d; w})})=d(\overline{G_{\D} \cdot (B \dot w, B) h_{J, J'}}) \\ &=\overline{d(G_{\D} \cdot (B \dot w, B) h_{J, J'})}=\overline{G_{\D} (B \dot w, B) h_{J'}} \\ &=\sqcup_{w' \in W^{J'}; w' \le_{J'} w} Z_{J', w'}.
%\end{align*}

%As a consequence, 

%(a) if $A \in D_G(Z_{J, \d})$ with support in $\overline{Z_{J, \d; w}}$, then the support of $f^J_{J'}(A)$ is contained in $\sqcup_{w' \in W^{J'}; w' \le_{J'} w} Z_{J', w'}$. 

%By definition, if $w$ is not a minimal length element in its $W_{J'}$-conjugacy class, then $\sqcup_{w' \in W^{J'}; w' \le_{J'} w} Z_{J', w'} \subset \sqcup_{w' \in W^{J'}; \ell(w')<\ell(w)} Z_{J', w'}$. 
\subsection{Induction functor} Let $(w, K)$ be a $\d$-combinatorial piece. Consider the following diagram 
\[
\xymatrix{
\frac{L_K}{\Ad_{\dot w \circ \d}(L_K)} \cong \frac{L_K \dot w}{\Ad_{\d}(L_K)} & \frac{P_K \dot w P_{\d(K)}}{\Ad_{\d}(P_K)} \ar[l]_-a \ar[r]^-b & \frac{G}{\Ad_\d(G)},
}
\]
where $a$ is induced from the projection map $P_K \dot w P_{\d(K)} \to L_K \dot w$ and $b$ is induced from the embedding $P_K \dot w P_{\d(K)} \to G$. 

Following \cite[\S 4.1]{S}, we define 
\begin{align*} \Ind^{(G, \d)}_{(L_K, \dot w)}=b_! a^*: \Sh(\frac{L_K}{\Ad_{\dot w \circ \d}(L_K)}) \to \Sh(\frac{G}{\Ad_\d(G)}). % \\ \quad \Res^{(G, \d)}_{(L_K, \dot w)}=a_! b^*: \Sh(\frac{G}{\Ad_\d(G)}) \to \Sh(\frac{L_K}{\Ad_{\dot w \circ \d}(L_K)}).
\end{align*}

If $w=1$, then $K=\d(K)$ and $P_K$ is a $\d$-stable standard parabolic subgroup of $G$. In this case, $\Ind$ is the Harish-Chandra induction functor. If $K=\emptyset$, then $P_K=B$ and $\Ind$ is the Deligne-Lusztig induction functor. 

Now we prove that 

\begin{theorem}\label{thm:main}
Let $J \subset J' \subset S$. Let $w \in {}^J W$ such that $w$ is of minimal length in $W_J \cdot_\d w$. Let $w' \in {}^{J'} W$, $u \in W_{I(J', w', \d)}$ and $x \in W_{J'} \cap W^{I(J', w', \d)}$ such that $x \i w \d(x)=u w'$. Set $K=I(J, w, \d)$, $K_1=\Ad(x) \i(K)$, $K'=I(J', w', \d)$ and $\d'=\Ad(\dot w') \circ \d$. Then we have the following commutative diagram
\[
\xymatrix{
\Sh(\frac{L_{K}}{\Ad_{\dot w \circ \d}(L_{K})}) \ar[d]_{\Ind^{(L_{K'}, \d')}_{(L_{K_1}, \dot u)} \circ \Ad(\dot x) \i} \ar[rr]^-{\pi^*_{J, \d; w}}_-{\cong} & & \Sh(\CY_{J, \d;w}) \ar[d]^{f^{J, w}_{J', w'}} \\
\Sh(\frac{L_{K'}}{\Ad_{\dot w' \circ \d}(L_{K'})}) \ar[rr]^-{\pi^*_{J', \d; w'}}_-{\cong} & & \Sh(\CY_{J', \d; w'}).
}
\]

%\[
%\xymatrix{
%\Sh(\frac{L_{K'}}{\Ad_{\dot w' \circ \d}(L_{K'})}) \ar[d]_{\Ad(\dot x) \circ \Res^{(L_{K'}, \d')}_{(L_{K_1}, \dot u)}} \ar[rr]^-{\pi^*_{J', \d; w'}}_-{\cong} & & \Sh(\CY_{J', \d;w'}) \ar[d]^{e_{J, w}^{J', w'}} \\
%\Sh(\frac{L_{K}}{\Ad_{\dot w \circ \d}(L_{K})}) \ar[rr]^-{\pi^*_{J, \d; w}}_-{\cong} & & \Sh(\CY_{J, \d;w}).
%}
%\]
\end{theorem}

\begin{proof}
We have the following commutative diagram
\[
\xymatrix{
\frac{L_{K}}{\Ad_{\dot w \circ \d}(L_{K})} \ar[dd]_{\Ad(\dot x) \i}^-\cong & \CY_{J, \d;w} \ar[l]_-{\pi_{J, \d; w}} \ar[r]^{i_{J, \d; w}} \ar@{}[dr]^-{\square} & \CY_{J, \d} \\
& q \i(\CY_{J, \d;w}) \ar[u]^-q \ar[d]_-f^-\cong \ar[r]^-i &  \CY_{J, J', \d} \ar[ddd]^p \ar[u]_q \\
\frac{L_{K_1}}{\Ad_{\dot u \circ \d'}(L_{K_1})} & X \ar[l]_-{\pi_1} \ar@{=}[d] & \\
\frac{(L_{K'} \cap P_{K_1}) \dot u \d' (L_{K'} \cap P_{K_1})}{\Ad_{\d'}(L_{K'} \cap P_{K_1})} \ar[u]^a \ar[d]_b \ar@{}[dr]^-{\square} & X \ar[l]_-{\pi_2} \ar[d]_{b'} & \\ \frac{L_{K'}}{\Ad_{\d'}(L_{K'})}
& \CY_{J', \d;w'} \ar[l]^-{\pi_{J', \d; w'}} \ar[r]_{i_{J', \d; w'}} & \CY_{J', \d}.
}
\]
Here $X=\frac{U_{P_{J'}} \backslash P_{K_1} \dot u \dot w' P_{\d(K_1)} /U_{P_{\d(J')}}}{\Ad_{\dot u \dot w' \circ \d}(L_{J'} \cap P_{K_1})}$. 

By Proposition \ref{prop:cyw}, $\CY_{J, \d;w} \cong \frac{U_{P_J} \backslash P_K \dot w P_{\d(K)} /U_{P_{\d(J)}}}{\Ad_{\dot w \circ \d}(L_J \cap P_K)}$. Thus by definition $q \i(\CY_{J, \d;w}) \cong \frac{U_{P_{J'}} \backslash P_K \dot w P_{\d(K)} /U_{P_{\d(J')}}}{\Ad_{\dot w \circ \d}(L_{J'} \cap P_K)}$. By Theorem \ref{thm:cyc}, $(w, K) \approx_{J, \d} (u w', K_1)$. The isomorphism $f: q \i(\CY_{J, \d;w}) \to X_1$ is given in Proposition \ref{prop:conj}. The map $\pi_1: X \to \frac{L_{K_1}}{\Ad_{\dot u \circ \d'}(L_{K_1})}$ is induced from the projection map $P_{K_1} \dot u \dot w' P_{\d(K_1)} \to L_{K_1} \dot u \dot w'$. The map $\pi_2: X \to \frac{(L_{K'} \cap P_{K_1}) \dot u \d' (L_{K'} \cap P_{K_1})}{\Ad_{\d'}(L_{K'} \cap P_{K_1})}$ is induced from the projection $$P_{K_1} \dot u \dot w' P_{\d(K_1)} \to P_{K_1} \dot u \dot w' P_{\d(K_1)} \cap L_{K'} \dot w'=(L_{K'} \cap P_{K_1}) \dot u \d' (L_{K'} \cap P_{K_1}).$$ By Proposition \ref{prop:cyw}, $\CY_{J', \d;w'} \cong \frac{U_{P_{J'}} \backslash P_{K'} \dot w' P_{\d(K')} /U_{P_{\d(J')}}}{\Ad_{\dot w' \circ \d}(L_{J'} \cap P_{K'})}$. The map $b': X \to \CY_{J', \d;w'}$ is induced from the embedding $P_{K_1} \dot u \dot w' P_{\d(K_1)} \subset P_{K'} \dot w' P_{\d(K')}$. 

Thus 
\begin{align*}
    f^{J, w}_{J', w'} \pi^*_{J, \d; w} &=i_{J', \d; w'}^* p_! q^* (i_{J, \d; w})_! \pi^*_{J, w, \d}=i_{J', \d; w'}^* p_! i_! q^* \pi^*_{J, w, \d}\\ 
    &=(b')_! f_! q^* \pi^*_{J, w, \d}=(b')_! \pi_1^* \Ad(\dot x) \i \\
    &=(b')_! \pi_2^* a^* \Ad(\dot x) \i=\pi_{J', \d; w'}^* b_! a^* \Ad(\dot x) \i \\
    &=\pi_{J', \d; w'}^* \Ind^{(L_{K'}, \d')}_{(L_{K_1}, \dot u)} \circ \Ad(\dot x) \i. 
\end{align*}
%Similarly, 
%\begin{align*}
%    e^{J', w'}_{J, w} \pi^*_{J', w', \d} &=i^*_{J, \d; w} q_! p^* (i_{J', \d, w'})_! \pi^*_{J', w', \d}=q_! i^* p^* (i_{J', \d, w'})_! \pi^*_{J', w', \d}\\ &=q_! f^* (b')^* \pi^*_{J', w', \d}=q_! f^* \pi_2^* b^*=q_! f^* \pi_1^* a_! [] \\ &=\pi^*_{J, w, \d} \Ad(\dot x) \Res^{(L_{K'}, \d')}_{(L_{K_1}, \dot u)} [].
%\end{align*} Here we use the fact that ???
The proof is finished. 
\end{proof}

\subsection{A special case} Now we discuss a special case of Theorem \ref{thm:main}. 

Let $J \subset J' \subset S$. Let $w \in {}^J W$ and $w' \in {}^{J'} W$ such that $w \in W_{J'} \cdot_\d w'$ and $\ell(w)=\ell(w')$. Set $K=I(J, w, \d)$ and $K'=I(J', w', \d)$. By Theorem \ref{thm:cyc}, there exists a unique $x \in W_{J'} \cap W^{K'}$ such that $x \i w \d(x)=w'$ and $(w, K) \approx_{J, \d} (w', K_1)$, where $K_1=\Ad(x) \i(K) \subset K'$. Since $\Ad(w) \d(K)=K$, we have $\Ad(w') \d(K_1)=K_1$. Set $\d'=\Ad(\dot w') \circ \d$. Then $L_{K_1}$ is a $\d'$-stable Levi subgroup and $L_{K'} \cap P_{K_1}$ is a $\d'$-stable parabolic subgroup of $L_{K'}$. The functor $\Ind^{(L_{K'}, \d')}_{(L_{K_1}, \dot u)}$ in Theorem \ref{thm:main} in this case is the Harish-Chandra induction functor. 

In this case, the functors $f^{J, w}_{J', w'}$ is just the Harish-Chandra induction functor, composed with the conjugation by $\dot x \i$. 

\subsection{Generalization to loop groups} Theorem \ref{thm:cyc} holds not only for Weyl groups but also for arbitrary Coxeter groups. The same proof works in such generality. In particular, Theorem \ref{thm:cyc} holds for any affine Weyl group together with a length-preserving automorphism on it. 

Let $\kk$ be an algebraically closed field and $F=\kk((\e))$ be the field of the formal Laurent series. Let $\bG$ be a connected reductive group over $F$ and $G=\bG(F)$. Let $I$ be an Iwahori subgroup of $G$ and $\tW$ be the Iwahori-Weyl group of $G$. Let $P \supset I$ be a parahoric subgroup. It is a pro-algebraic group. Let $U_P$ be the pro-unipotent radical of $P$ and $L \cong P/U_P$ be its reductive quotient. In \cite{Lu3}, Lusztig introduced the affine analogy of the parabolic character sheaves. They are certain simple perverse sheaves on the ind-stack $\frac{U_P \backslash G/U_P}{\D(L)}$. 

Note that the Iwahori-Weyl group $\tW$ is not a Coxeter group in general, but a quasi-Coxeter group. Namely, let $G^0$ be the subgroup of $G$ generated by all the parahoric subgroups. The Iwahori-Weyl group of $G^0$ is an affine Weyl group $W_a$. And $\tW=W_a \rtimes \Omega$, where $\Omega$ is the subgroup of length-zero elements of $\tW$. The conjugation action of any element in $\Omega$ on $W_a$ is a length-preserving automorphism. We have $$\frac{U_P \backslash G/U_P}{\D(L)}=\bigsqcup_{\t \in \Omega} \frac{U_P \backslash G^0 \dot \t/U_P}{\D(L)} \cong \bigsqcup_{\t \in \Omega} \frac{U_P \backslash G^0/U_{\d_{\dot \t }(P)}}{\Ad_{\d_{\dot \t}}(L)}.$$ Here $\d_{\dot \t}$ is the automorphism on $G^0$ given by the conjugation action of $\dot \t$. Now Theorem \ref{thm:main} may be applied to $\frac{U_P \backslash G^0/U_{\d_{\dot \t }(P)}}{\Ad_{\d_{\dot \t}}(L)}$ using Theorem \ref{thm:cyc} for the pair $(W_a, \Ad(\t))$.

\end{document}